\documentclass[12pt,reqno]{article}
\usepackage{geometry}
\usepackage{amsmath,amsthm,amssymb,mathrsfs,bm,mathtools}
\usepackage[mathlines]{lineno}
\usepackage{latexsym}

\usepackage{graphicx,booktabs,multirow}
\usepackage{tikz}
\usepackage{appendix}
\usepackage{yhmath}
\usepackage{float}
\usepackage{color}
\usepackage{subfigure}
\usepackage{rotating}
\usepackage[section]{placeins}
\usepackage{setspace}
\usepackage{pdfpages}

\newtheorem{theorem}{Theorem}
\newtheorem{lemma}{Lemma}

\newtheorem{observation}{Observation}
\newtheorem{problem}{Problem}
\newtheorem{proposition}{Proposition}

\newtheorem{claim}{Claim}

\newtheorem{property}{Property}

\usepackage[colorlinks,
linkcolor=red!60!black,
anchorcolor=blue,
citecolor=blue!60!black
]{hyperref}
\usepackage{url}

\numberwithin{equation}{section}

\usepackage{pdfpages} 
\usepackage{caption}

\usetikzlibrary{backgrounds}
\usetikzlibrary{arrows}
\usetikzlibrary{shapes,shapes.geometric,shapes.misc}

\begin{document}

 \captionsetup[figure]{labelfont={bf},name={Fig.},labelsep=period}
	\title{ On the sizes  of generalized cactus graphs 
	\thanks{The work is supported by the National Natural 
	Science Foundation of China (Grant No. 
	12271157) and the Natural Science Foundation of Hunan Province, China (Grant 
	No. 2022JJ30028).}}

\author
{Licheng Zhang\thanks{Email: lczhangmath@163.com.},
Yuanqiu Huang\thanks{Corresponding author. Email: hyqq@hunnu.edu.cn.}\\
\small College of Mathematics and Statistics\\
\small Hunan Normal University, Changsha 410081, P.R. China\\
}

\date{}
\maketitle
	
\begin{abstract}
A cactus is a connected graph in which each edge is contained in at most one 
cycle.  We generalize the concept of cactus graphs, i.e.,  a
$k$-cactus is a connected graph in which each edge is contained in at most 
$k$ cycles where $k\ge 1$.  It is well known that every 
cactus with $n$ vertices has at most $\lfloor\frac{3}{2}(n-1) \rfloor$ edges. 
Inspired by it, we attempt to 
establish analogous upper bounds for general $k$-cactus graphs. In this
paper, we first characterize $k$-cactus graphs for $2\le k\le 
4$ based on the block decompositions. Subsequently, we give  tight 
upper bounds on their sizes. Moreover, the corresponding extremal graphs are 
also characterized. However, for larger $k$, this extremal problem remains 
open. For the case of 2-connectedness, the range of $k$ is expanded to all 
positive integers in our research.  We prove that every $2$-connected $k ~(\ge 
1)$-cactus graphs with $n$ vertices has at most $n+k-1$ edges, and the bound is 
tight if $n \ge k + 2$. 
		


\end{abstract} 

\section{Introduction}

Only finite simple graphs are considered.  The {\it order} of a graph is its 
number of vertices and the {\it size} is its 
number of edges.  A {\it tree }  is a connected 
graph that contains no cycles, and a {\it cactus} (also known as 
 Husimi tree \cite{Harary}) is a connected graph in which each edge is 
 contained in at most one cycle.  Obviously, the cactus family includes trees. 
 Cactus graphs are a well-studied  graph class, 
 both in terms of 
 their structure and algorithms. For example,  the 
 family of cactus graphs is closed under graph minor operations, and they can 
 be  characterized by a single forbidden minor, the  diamond graph 
 obtained by removing an edge from the complete graph $K_4$ \cite{El-Mallah}. 
 Some  combinatorial optimization problems, 
 like facility location problems and finding the fast search number, which 
 are NP-hard for general graphs, can be solved in polynomial time 
 for  cactus graphs \cite{Ben-Moshe,Xue2023}. It is worth noting that although 
 the structure of cactus graphs is relatively simple, some problems remain 
 unsolved. One notable example is the well-known Rosa's conjecture \cite{Rosa}.
 For more properties of cactus graphs, readers may refer to \cite{Araujo, Czap, 
 Li2019, Sedlar}.

 From the perspective of enumerating the number of cycles containing
edges, cactus graphs can be extended further. We say that a connected 
graph is a {\it $k$-cactus} if each edge is contained in at most $k$ cycles
 where $k\ge 1$. In particular,  1-cactus graphs 
are cactus graphs. While the concept ``$k$-cactus" may seem novel, there are
early results on analogous cycle-edge patterns. For example, in 1973, 
Toida  \cite{Toida1973} proved that a graph $G$ is an Eulerian graph if and 
only if every edge of  $G$ lies on an odd number of cycles. In 1988, Erd\H{o}s 
\cite{Erdos} established that  $|E(G)| 
\geq \frac{3}{2}\left(|V(G)|-1\right)$ if each edge of graph $G$ lies in at 
least one triangle (i.e., $C_3$).

Clearly, every tree with $n$ vertices has exactly $n-1$ edges. The 
following well-known theorem gives an upper bound on the sizes of cactus 
graphs. The theorem is explicitly documented in some books, such as the 
textbook \cite{west2001}, or implied in \cite{Lovasz} (refer to the first 
exercise in page 73), yet its original authorship remains a mystery so far.

\begin{theorem}[p.160, \cite{west2001}]
Let $G$ be a cactus with $n$ vertices. Then 
$|E(G)|\le\left\lfloor{\frac{3(n-1)}{2}}\right\rfloor$.  The bound is achieved 
by a set of $  \lfloor (n -1)/2 \rfloor $ triangles sharing
a single vertex, plus one extra edge to a leaf if $n$ is even.
\end{theorem}

Motivated by Theorem 1, we pose the following natural problem on $k$-cactus 
graphs where $k\ge 2$.

\begin{problem}\label{p1}
Let $G$ be a $k$-cactus with $n$ vertices where $k\ge 2$. Can we establish a  
(tight) upper bound $f(n,k)$ on the sizes of $G$? 
\end{problem}
For the case $k = 2, 3, 4$, we  answer the above problem.  Moreover, the 
corresponding extremal graphs are also 
characterized. But for $k\ge 5$ (especially for larger or more general values 
of $k$), the problem remains open.

The paper is structured as follows. In next section, we introduce some
terminology, definitions, and a few  preliminary lemmas. In Section 3,  we
characterize  $k$-cactus graphs for $2\le k\le 4$ in terms of the block 
decompositions. In Section 4,  we first prove that every 
$2$-connected $k ~(\ge 1)$-cactus graphs with 
$n$ vertices has at most $n+k-1$ edges, and the bound is tight if $n \ge k + 
2$. Next we give  tight upper bounds on the sizes of  $k$-cactus graphs where 
$2\le k\le 4$, and  the corresponding extremal graphs are also characterized. 
We conclude in Section 5 with a discussion of some open questions.

\section{Preliminaries}\label{}

 Let $G$ be a graph.  Let $v$ be a vertex of  $G$. We denote by $N_G(v)$ the 
 set of vertices adjacent to $v$,  and by $I_G(v)$  the set of edges incident 
 with $v$. A {\it component} of $G$ is a connected subgraph 
 $H$ such that no subgraph of $G$ that properly contains $H$ is connected.   A 
 {\it cut-vertex} is a vertex whose removal increases the number of
 components.   We write $G-v$ or $G-S$ for 
 the subgraph obtained by deleting a vertex $v$ or set of vertices $S$. An 
 {\it induced subgraph} is a subgraph obtained by deleting a set of vertices. 
 We  write $G[T]$ for $G-\bar{T}$, where $\bar{T}=V(G)\setminus T$. A {\it 
 block} of $G$ 
 is a maximal connected 
  subgraph of $G$ that has no cut-vertex.  A block is called a {\it $H$-block} 
  if the block is  isomorphic  to the graph $H$. 
 
For a non-negative integer $n$, we denote $K_n$ a complete graph with $n$ 
vertices, and if $n\ge 3$ we denote by $C_n$ a cycle with $n$ vertices.
The {\it union} of graphs $G$ and $H$ is the graph $G \cup H$
with vertex set $V (G) \cup V (H)$ and edge set $E(G) \cup E(H)$.

The subsequent observation follows directly from the definition of blocks.

 \begin{observation}\label{block}
 Let $G$ be a connected graph with at least $2$ vertices. If a block $H$ of $G$ 
 is not a  single edge, then $H$  is $2$-connected.
 \end{observation}

  A {\it decomposition} of a graph is a list of subgraphs such that each edge 
  appears in exactly one subgraph in the list.
\begin{lemma}[ \cite{west2001}]\label{blockdec}
Let $G$ be a graph. Then the blocks of $G$ form a decomposition of $G$.
\end{lemma}

Let $G_1$ and $G_2$ be two vertex-disjoint graphs, and let $v_1\in V(G_1), 
v_2\in V(G_2)$. The {\it coalescence} of $G_1$, $G_2$ with respect to $v_1, 
v_2$, denoted by $G_1(v_1)\diamond G_2(v_2)$, is obtained from $G_1$, $G_2$ by 
identifying $v_1$ with $v_2$ and forming a new vertex $u$.  In other 
words, $V(G_1(v_1) \diamond  G_2(v_2))=V(G_1) \cup 
V(G_2) \cup \{u\}\setminus \{v_1, v_2\}$, and $E(G_1(v_1) \diamond  G_2(v_2))= 
\big( 
E(G_1)\setminus I_{G_1}(v_1)\big) \bigcup \big( E(G_2) \setminus 
I_{G_2}(v_2)\big) 
\bigcup \big( \{ux|x\in N_{G_1}(v_1)\cup N_{G_2}(v_2) \big)\}$. If the 
identifying 
vertices are not specified, the coalescence of $G_1$ and $G_2$ can be denoted 
as $G_1\diamond G_2$. For a series of  pairwise disjoint graphs 
$G_1,G_2,\dots, G_t$ where $t\ge 3$, we define their  coalescence as  
$G_1\diamond G_2\cdots \diamond G_t$ where $G_1\diamond G_2\cdots \diamond 
G_{i-1}\diamond G_{i} =(G_1\diamond G_2\cdots 
\diamond G_{i-1})\diamond G_{i}$ for every $ 2\le i\le t$.

Let $G_1,G_2,\dots, G_t$ be pairwise disjoint graphs. Let $\xi$ be a 
graph class that consists of each graph that is a coalescence of 
$G_1,G_2,\dots, G_t$. 

\begin{lemma}\label{edgesame}
Let $G$ be a graph in $\xi$. Then $|V(G)|=\sum_{i=1}^k |V(G_i)|-(t-1)$ and 
$|E(G)|=\sum_{i=1}^k |E(G_i)|$.
\end{lemma}

\section{Characterization of  \boldmath{$k$}-cactus graphs for \boldmath{$2 \le 
k\le 
4$}}

Clearly, Lemma \ref{blockdec} implies that any graph has a block 
decomposition.
Therefore, if we clarify structures of each block of a graph, the 
 characteristics of the entire 
graph will also become clearer. It is well known that there are several 
equivalent definitions of cactus graphs, one of which is that  a graph is a 
cactus if and only if  each of its blocks  is either a cycle or an edge. 
Similarly, we provide a characterization of $k$-cactus graphs for $2\le k\le 
4$ based on the block decompositions.

By Observation \ref{block}, any block with at least 3 vertices is 2-connected. 
Therefore, some specialized tools for 2-connected graphs are needed.
Let $F$ be a subgraph of $G$.
An {\it ear} of $F$ in $G$ is a nontrival path in $G$ whose vertex-ends lie in 
$F$ but whose internal vertices 
do not. First, we introduce the concept of the ear decompositions. An {\it ear 
decomposition} of $G$ is a decomposition $Q_0, Q_1, \cdots, Q_k$ such that 
$Q_0$ is a cycle and $ Q_i$ for $i\ge 1$ is an ear of $Q_0 \cup Q_1  \cdots 
\cup Q_{i-1}$. We denote the ear decomposition in the mentioned order as $[Q_0, 
Q_1, ..., Q_l]$.

\begin{lemma}[p.162, \cite{west2001}]\label{addear}
Let $G$ be a $2$-connected graph. Then $G\cup Q$ is $2$-connected where $Q$ is 
an ear of $G$ in  $G\cup Q$. 
\end{lemma}

The following lemma, known as the ``ear decomposition theorem'', gives a 
technique to construct a 2-connected graph from
a cycle and paths. It will be a key 
tool for characterizing  2-connected $k$-cactus graphs for $2\le k\le 4$. 
\begin{lemma}[Whitney, \cite{whitney}]\label{eardecom}
A graph is $2$-connected if and only if it has an ear decomposition.
\end{lemma}

\begin{lemma}[p.162, \cite{west2001}]\label{2_c_ed}
Any two edges are contained in a cycle in a $2$-connected graph.
\end{lemma}

\begin{lemma}\label{addears}
 Let $G$ be a $2$-connected graph and let  $uv$ be an edge. Let $Q$ be an ear 
 of $G$ in $G \cup Q$.  If $uv$ is contained in $k$ cycles in $G$, then $uv$  
 is contained  in at least $k + 1$ cycles in $G \cup Q$.
\end{lemma}

\begin{proof}
 First,  the $k$ different cycles in $G$ that contain $uv$ are also 
 preserved  in $G\cup Q$ (i.e., they are also $k$ distinct cycles 
 that  contain $uv$ in $G\cup Q$).   Additionally, by Lemma \ref{addear}, $G 
 \cup Q$ is  still 2-connected, and thus $uv$ and an edge in $Q$ will lie in a 
 common cycle $C'$ in $G\cup Q$ by Lemma \ref{2_c_ed}. It is clear that 
 $C'$ is distinct from  any of the $k$ cycles containing $uv$,  as desired.
\end{proof}

 A {\it $k$-theta graph} denoted by $\theta(l_1,l_2,\dots,l_k)$ where $k\ge 
3$ is the graph  obtained from $k$ internally  disjoint paths of lengths 
 respectively $l_1,l_2,\dots, l_k$, with the same pair of end-vertices. 
Notably, among these $k$ internally disjoint paths, there exists at most one 
path with length one, since we only consider simple graphs. Obviously,  
3-theta 
 graphs are our familiar theta 
 graphs. For short, we refer to a $k$-theta graph as a $\theta_k$-graph.
We denote $\theta'(p,q,r)$  as a graph obtained by adding an ear to a theta 
graph $\theta(p,q,r)$  where both end-vertices of the ear are not the
vertices with degree 3 of $\theta(p,q,r)$.  For short,  we denote the graph  
$\theta'(p,q,r) $ by $\theta_3'$-graph.

We call a $k$-cactus $H$ is {\it nice} if  there exists an edge in $H$
contained in  exactly $k$ cycles.

\begin{observation}\label{heritability}
Let $k$ and $l$ be integers where $0\le k < l$. Then any $k$-cactus is an 
$l$-cactus.
\end{observation}

\begin{lemma}\label{nicecactus}
(i) Every cycle is a nice cactus. (ii) Every  $\theta_k$-graph is
 a nice $k-1$-cactus.(iii) Every $\theta_3'$-graph is a nice $4$-cactus.
\end{lemma}

\begin{proof}
It can verify (i) and (ii) directly, so we  only prove (iii).  
By definition of $\theta_3'$-graph, it is obtained by adding an ear 
$Q=q_0q_1\dots q_s$ 
from a theta graph $\theta (l_1,l_2,l_3)$. Let $u$ and $v$ be the unique 
two vertices with degree  $3$ in  $\theta (l_1,l_2,l_3)$. Let 
$P^1=uu_1'u_2'\dots u_{l-1}' v$, $P^2=uu_1''u_2''\dots u_{l-1}'' v$, 
and $P^3=uu_1'''u_2'''\dots u_{l-1}''' v$ be three internally vertex 
disjoint paths connecting  $u$ and $v$ where $|E(P^i)|=l_i$ for $1\le i \le 3$.

We will discuss three cases regarding the placement of the two end-vertices of 
$Q$ at the position of $\theta (l_1,l_2,l_3)$.

\noindent \textbf{Case 1.}  
$\{q_0,q_s\}\subseteq V(P^i) $ and $\{q_0,q_s\}\cap \{u,v\}=\emptyset$ where  
$1\le i\le 3$. Without loss of generality, we assume  $i = 3$, $q_0=u_c'''$ and 
$q_s=u_{c+m}'''$ where $c$ and $m$ are integers and $c\ge 1$ and $m\ge 1$; see 
Fig. \ref{fig1} (1)  for an illustration. Let $P^a=u_c'''(=q_0)u_{c+1}'''\dots 
u_{c+m-1}'''u_{c+m}'''(=q_s)$. Let $P^b$ be the path obtained by replacing the 
subpath $P^a$ with $Q$ in $P^3$.
Now, it is easy to verify that any 
edge on  $P^i$  for $i=1,2$ will be contained in exactly three cycles: $P^i\cup 
P^{3-i}$, $P^i\cup P^3$, $P^i \cup P^b$.
Clearly, $Q$ and $P^a$ are symmetric. Hence we only consider $Q$, and  observe 
that any edge in $Q$ is contained in exactly
three cycles $Q\cup P^a$, $P^1\cup P^3$, and 
$P^2\cup P^b$. We  observe that  any edge of two paths  $uu_1'''\dots 
u'''_{c-1}q_0 $ or $q_su_{c+m+1}'''\dots u'''_{l_3-1}v$
 will be contained in exactly four cycles $P^1\cup P^3$, 
 $P^2\cup P^3$, $P^1\cup P^4$ and $P^2\cup P^4$.

\noindent \textbf{Case 2.} $\{q_0,q_s\}\subseteq V(P^i) $ and $|\{q_0,q_s\}\cap 
\{u,v\}|=1$ where  $1\le i\le 3$.  This case is similar to Case 1, and we  omit 
its proof. See also Fig. \ref{fig1} (2) for an illustration.

\noindent \textbf{Case 3.}  $q_0 \in  V(P^i)\setminus \{u,v\} $ and $q_s \in  
V(P^j)\setminus \{u,v\}$ where  $1\le i,j\le 3$ and $i\ne j$. In this case, 
without loss of generality, 
we assume  $i = 2$ and $j=3$; as shown in Figure 1 (3). The vertices $q_0$ and 
$q_s$ divide $P^2$ and $P^3$ into two subpaths, respectively.
Let $P^2_a=uu_1''u_2''\dots q_0$, $P^2_b=q_0\dots u'''_{l_2-1}v$, 
$P^3_a=uu_1'''u_2'''\dots q_s$, $P^3_b=q_s\dots u'''_{l-1}v$. Considering 
symmetry, we only need to determine the number of cycles containing any edge 
on $P^1$, $Q$, $P^2_a$, respectively. For each edge in $P^1$,  it is 
contained in exactly four cycles $P^1\cup P^2$, $P^1\cup P^3$, $P^1\cup 
P^2_a\cup Q\cup P^3_b$ and $P_1\cup P^2_b\cup Q \cup P^3_a$.
For each edge in $Q$, it is contained 
in exactly four cycles $Q\cup P^2_a\cup P^3_a$, $Q\cup P^2_b\cup P^3_b$, 
$Q\cup P_b^3\cup P^1 \cup P^2_a$ and  $Q\cup P^2_b \cup P^1\cup P^3_a $. 
As for  each edge in $P^2_a$, it is contained 
in exactly four cycles $P^1\cup P^2$, $P^2\cup P^3$, $P^2_a\cup Q \cup P^3_a$ 
and 
$P^2_a\cup Q \cup P^3_b\cup P^1$.

By above discussion, every $\theta_3'$-graph is a nice $4$-cactus.
\end{proof}

\begin{figure}[H]
	\centering
	\includegraphics[scale=0.45]{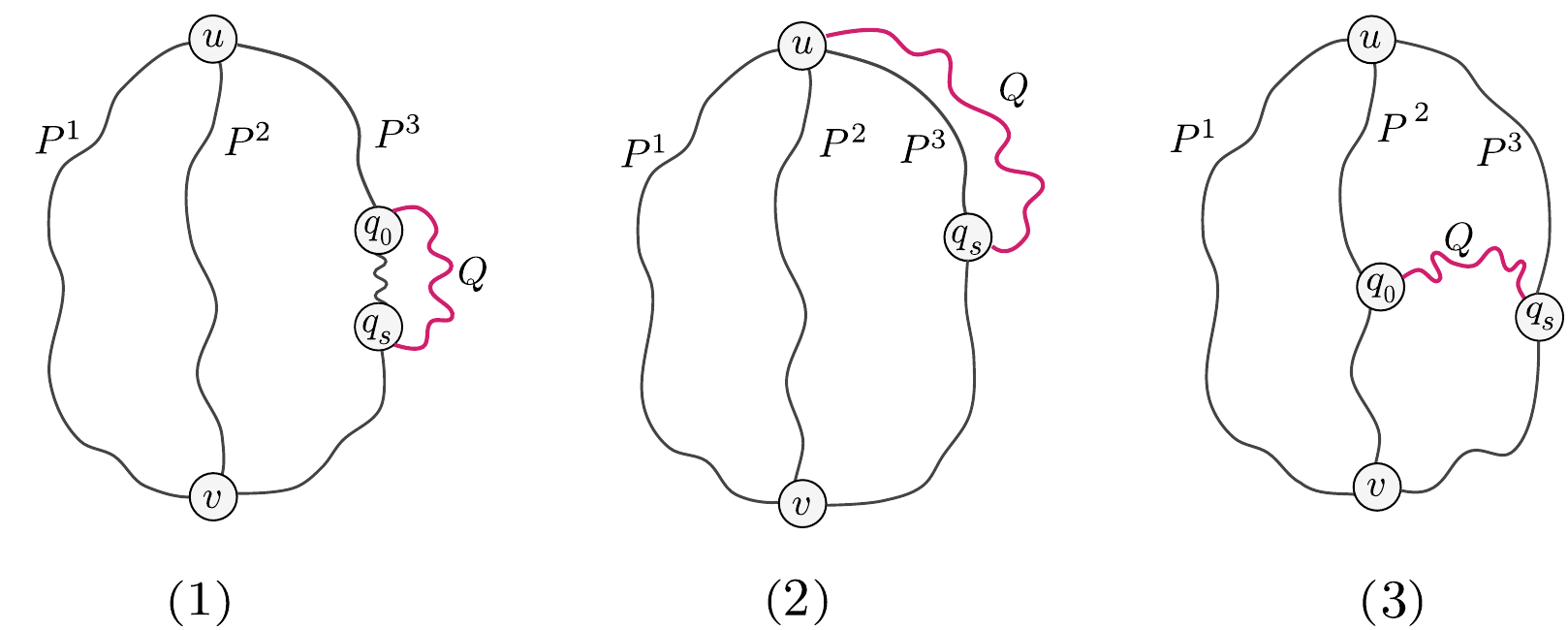}
	\caption{ Three ways of adding the ear $Q$ to $\theta (l_1,l_2,l_3)$.}
	\label{fig1}
\end{figure}

\begin{lemma}\label{2_earadd}
Let $G$ be a $2$-connected $k$-cactus where $k\ge 1$. Let $[Q_0, Q_1,\cdots, 
Q_l]$ be an ear decomposition of $G$ where $Q_0$ is a cycle. Then $l\le k-1$.
\end{lemma}

\begin{proof}
Let $e$ be an edge of the cycle $Q_0$. Then $e$ is contained  in one cycle. 
By repeatedly applying Lemma \ref{addears},  $e$ will  be contained in at least 
$l+1$ cycles in   $Q_0\cup Q_1 \cdots \cup Q_l$.  Since $G$ is a $k$-cactus, we 
have $k\le l+1$, yielding $l\le k-1$. 
\end{proof}

\begin{theorem}\label{cactus_stru}
Let $G$ be a graph. Then the following statements hold.
\begin{itemize}
\item [(i)] $G$ is a $2$-cactus if and only if each block of $G$  is either 
an edge, a cycle, or a $\theta_3$-graph.
\item [(ii)] $G$ is a $3$-cactus  if and only if each block of $G$  is 
either an edge, a cycle, or a   $\theta_t$-graph where  $3\le t\le 4$.
\item [(iii)] $G$ is a $4$-cactus if and only if each block of $G$  is 
either an edge,  a cycle,  a $\theta_3'$-graph, a   $\theta_t$-graph where  
$3\le t\le 5$.
\end{itemize}
\end{theorem}

\begin{proof}

Clearly,  ``if"-parts of the three statements (i), (ii) and (iii) follow from 
Lemma \ref{nicecactus} and Observation \ref{heritability}.  Now we prove their 
``only if"-parts, respectively. Let $B$ be a block of $G$. 

For (i), it is evident that $B$ can be an edge. If $B$ is not an edge, then $B$ 
is 
2-connected. Then by Lemmas \ref{eardecom} and \ref{2_earadd}, 
$B$ has an ear decomposition and at most one ear in the decomposition. So $B$ 
is either a cycle or a theta graph, as desired.

For (ii),  clearly, $B$ can be an edge. If not, by Lemma \ref{eardecom}, we can 
assume that $[Q_0,\cdots, Q_l]$ is an ear decomposition of $B$. Furthermore, by 
we have  $l\le 2$ by Lemma \ref{2_earadd}. For $l=0$, then $B$ is a cycle. As 
for $l=1$, $B$ is a theta graph. For $l=2$, then by Lemma 
\ref{nicecactus} (ii)-(iii), $B$ must be a $\theta_4$-graph. 

For (iii),  clearly, $B$ can be an edge. If not, we 
can  similarly assume that $[Q_0,\cdots, Q_l]$ is an ear decomposition of $B$. 
By Lemma \ref{2_earadd}, we have $l\le 3$. For $l\le 2$, then $B$ can be either 
a 
cycle, a theta 
graph, or a $\theta'$-graph (by Lemma \ref{nicecactus} (iii)). As for $l=3$, 
its subgraph $Q_0\cup Q_1\cup Q_2$ is a $\theta_4$-graph or a 
$\theta_3'$-graph. 
Furthermore, if $Q_0\cup Q_1\cup Q_2$ is a 
$\theta_3'$-graph, then by Lemma \ref{nicecactus} (iii), it is known to be a 
nice 4-cactus, which means there exists an edge that lies on 4 cycles.
Then, by Lemma \ref{addears} there exists at least one edge 
in $Q_0\cup Q_1\cup Q_2\cup Q_3$ (i.e. $B$)  being contained $5$-cycles , which 
contradicts the assumption. So 
$Q_0\cup Q_1\cup Q_2$ is a  $\theta_4$-graph. Let $P^1$, $P^2$, $P^3$, $P^4$ be 
four  internally disjoint  
 paths in $Q_0\cup Q_1\cup Q_2$  that connect two (unique)  vertices of degree 
 $4$. We claim that both two end-vertices of $Q_3$ must be the 
 distinct vertices of degree 4. Otherwise, combining the assumption of $Q_0\cup 
 Q_1\cup Q_2$   with the definition of ears,  we can obtain a $\theta_3'$-graph 
 by deleting some path $P^i$ from $B$ where $1\le i\le 4$,  Based on the 
 previous 
 discussion, adding back $P^i$, there is an edge of $B$ that lies on five 
 cycles based on the earlier discussion, a contradiction. This complete the 
 proof.
\end{proof}

\section{Upper bounds and extremal graphs}

Before proving the main theorem, we give three simple lemmas without proof.
\begin{lemma}\label{f0}
Let $G$ be a $\theta_k$-graph. Then  $|V(G)|\ge k+1$ and $|E(G)| =n+k-2$.

\end{lemma}

\begin{lemma}\label{f1}
Let $G$ be a $\theta'$-graph. Then  $|V(G)|\ge 4$ and $|E(G)| =n+2$.
\end{lemma}

\begin{lemma}\label{f2}
Let $x$ and $y$ be two reals. Then $\lfloor x\rfloor+\lfloor y\rfloor 
\leq \lfloor x+y\rfloor$.
\end{lemma}

Now we  investigate some upper bounds on the sizes for  2-connected 
$k$-cactus graphs.

\begin{proposition}\label{2c_k}
Let $G$ be a $2$-connected $k$-cactus with $n$ vertices where  $k\ge 1$. Then
 $|E(G)|\le n+k-1$. Moreover, the bound is tight for $n\ge k+2$.
\end{proposition}

\begin{proof}
 By Lemma \ref{eardecom},  we can assume  that $[Q_0,Q_1,\cdots, Q_r]$ is an 
 ear decomposition  of $G$. Given the definition
of ears, $n=  |V(Q_0)|+ 
\sum_{i=1}^r (|V(Q_i)|-2)$ and 
$|E(G)|=|E(Q_0)|+\sum_{i=1}^r|E(Q_i)| 
=|V(Q_0)|+\sum_{i=1}^r(|V(Q_i)|-1)=n+r$. From Lemma \ref{2_earadd}, it follows 
that  $r\le k-1$, and thus $|E(G)|\le n+k-1$. Moreover, by Lemma 
\ref{nicecactus} (ii), every $\theta_{k+1}$-graph with $n$ is a 
$k$-cactus.  By Lemma \ref{f0}, it has at least $k+2$ vertices and exactly 
$n+k-1$ edges. So the bound is tight for $n\ge k+2$.
\end{proof}

The bound in Proposition \ref{2c_k} may not be tight when $n \le k-1$. In the 
following, we obtain some possible better bounds for cases where $n$ takes very 
small values.

\begin{proposition}\label{2c_small}
Let $G$ be a $2$-connected $k$-cactus with $n$ vertices where  $k\ge 1$. Then
\begin{itemize}
\item [(i)] for $n=3$ and every $k$, then $G\cong K_3$ and $|E(G)|=3$.
\item [(ii)] for $n=4$, if $k=2,3$ then $|E(G)|\le 5$, with equality  if and 
only if $G\cong \theta(1,2,2)$; if $k=4$, then $|E(G)|\le 6$, with equality  if 
and only if $G\cong K_4$; 
\item [(iii)] for $n=5$, if $k=2$, then  $|E(G)|\le 6$, with equality  if and 
only if $G\cong\theta(1,2,2)$ or $\theta (1,2,3)$; if $k=3$, 
$|E(G)|\le 7$,  with equality  if and 
only if $G\cong \theta(1,2,2,2)$; if $k=4$, then $|E(G)|\le 7$  
with equality  if and 
only if $G \cong \theta(1,2,2,2)$, $\tilde{\theta}_1$ or  
$\tilde{\theta}_2$, where $\tilde{\theta}_1$ and  $\tilde{\theta}_2$ are shown 
in Fig. \ref{fig2} 
(b)-(c), respectively.
\end{itemize}
\end{proposition}

\begin{proof}
(i). It is obvious.

(ii). By Theorem  \ref{cactus_stru} (ii) and (iii), $K_4$ is a nice 
$4$-cactus  which means that $K_4$ is not a $k$-cactus where $k<4$. Moreover, 
it is easy to see that $\theta(1,2,2)$ is the unique $2$-cactus with $5$ edges. 
So  (ii) can be derived directly.

(iii). By Proposition \ref{2c_k}, $E(G)\le 6$ if $k=2$ or $E(G)\le 7$  for 
$k=3$. It is easy to verify that  graphs with sizes achieving  the equalities 
are $\theta(1,2,2)$ for $k=2$ and $\theta(1,2,2,2)$ for $k=3$, respectively.
 For $k=4$,  as every $\theta_4$-graph has at least $6$ 
vertices,  $G$ can be a cycle, or a 
$\theta_3'$-graph by  Theorem \ref{cactus_stru} (iii). Combining this with 
Lemmas 
\ref{f0} and \ref{f1}, we have $|E(G)|\le 7$.
One can check that  a 2-connected 4-cactus with 5 vertices and 7 edges is 
either $\theta(1,2,2,2)$, $\tilde{\theta}_1$, or  $\tilde{\theta}_2$. 
\end{proof}

\begin{figure}[H]
	\centering
	\includegraphics[scale=1.2]{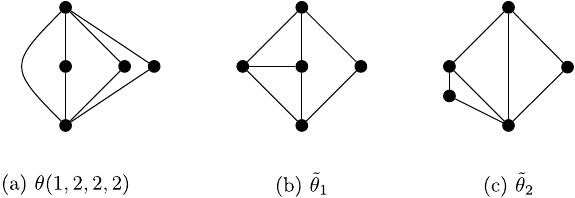}
	\caption{All 2-connected $4$-cactus graphs with $5$ vertices and $7$ edges.}
	\label{fig2}
\end{figure}

Now we prove the main theorems.

\begin{theorem}\label{2or3cactus}
Let $G$ be a $k$-cactus  with $n$ vertices where $2\le k\le 3$. Then 
$|E(G)|\le \left\lfloor \frac{(2k+1)(n-1)}{k+1}\right\rfloor$.  Moreover, the 
bound is tight.
\end{theorem}

\begin{proof} 
We apply induction  on the 
number of blocks of $G$.   Let $l$ be the number of
blocks of $G$. Clearly, if $l=1$, then $|E(G)|= n-1$ if $n \leq 2$.  
 If $n=3$, then $|E(G)|= 3$ by Proposition 
\ref{2c_small} (i). If $n=4$, then $|E(G)|\le 5$ by Proposition \ref{2c_small} 
(ii). If $n\ge 5$, then $|E(G)|\le 
n+k-1$ by  Proposition \ref{2c_k}. In either case, 
$|E(G)|\leq\lfloor (2k+1)(n-1) / k \rfloor$ for $k=2,3$.

A graph that has more than one block is not a single block, so it has
a cut-vertex $v$. Let $S$ be the vertex set of one component of $G-v$. A graph 
that has more than one block is not a single block, so it has a cut-vertex $v$. 
Let $S$ be the vertex set of one component of $G-v$. Let $G_1=G[S \cup\{v\}]$, 
and let $G_2=G-S$. Both $G_1$ and $G_2$ are $k$-cactus graphs, and every block 
of $G$ is a block in exactly one of $G_1$ and $G_2$. Thus each has fewer 
blocks than $G$, and we can apply the induction hypothesis to obtain 
$|E\left(G_i\right)| \leq$ $\left\lfloor 
(2k+1)\left(|V\left(G_i\right)|-1\right) / 
(k+1)\right\rfloor$ for $i=1,2$. Since $v$ belongs to both graphs $G_1$ and 
$G_2$,  
we have 
$|V(G_1)|+|V(G_2)|=n+1$.  We thus have
\begin{linenomath}
\begin{align}
|E(G)|&=|E\left(G_1\right)|+|E\left(G_2\right)| \notag\\
&\leq\left\lfloor\frac{(2k+1)(|V(G_1)|-1)}{k+1}\right\rfloor+\left\lfloor
\frac{(2k+1)(|V(G_2)|-1)}{k+1}\right\rfloor  \notag \\
&\leq\left\lfloor\frac{(2k+1)(|V(G_1)|-1)}{k+1}+\frac{(2k+1)(|V(G_2)|-1)}{k+1}\right\rfloor
 \label{eq1} \\
&=\left\lfloor\frac{(2k+1)(n-1)}{k+1}\right\rfloor, \notag
\end{align}
\end{linenomath}
in which the inequality (\ref{eq1}) is follows from Lemma \ref{f2}.

We delay explaining the tightness of the bound until Theorems \ref{2ex} and 
\ref{3ex}.
\end{proof}

Here, one might wonder whether the bound in Theorem \ref{2or3cactus} still 
holds for  $k\ge 4$. Regrettably, the answer may be negative. For example, it 
is no 
longer correct when $k=4$; see Theorem \ref{4cactus}. 

\begin{theorem}\label{4cactus}
Let $G$ be a $4$-cactus with $n$ vertices. Then 

\begin{linenomath}
$$|E(G)|\le\begin{cases}
2n-2 &\text{if}~  n \equiv  1 \pmod 3, \\
2n-3 &\text{otherwise.} \\
\end{cases}$$
\end{linenomath}
 Moreover, these bounds are tight.
\end{theorem}
\begin{proof}

We also use induction on the number of blocks of $G$ (which is similar to 
Theorem \ref{2or3cactus} except for some minor details). Let $l$ be 
the number of blocks of $G$. We first consider the case when $l=1$. If $n \leq 
2$  $|E(G)|= n-1$ and  $|E(G)|= 3$ if $n=3$. By Proposition \ref{2c_small} 
(ii)--(iii),  $|E(G)| 
\le 6$ if $n=4$ and $|E(G)| \le 7$  if $n =5$. As for $n\ge 6$, we have $|E(G)| 
\le n+3$ by Proposition \ref{2c_k}. In either case, $|E(G)|\leq 2n-2$ when $n 
\equiv  1 \pmod 3$ and $|E(G)|\leq 2n-3$ when $n \equiv  0 ~\text{or}~ 2 \pmod 
3$.

A graph that has more than one block is not a single block, so it has
a cut-vertex $v$. Let $S$ be the vertex set of one component of $G-v$. A graph 
that has more than one block is not a single block, so it has a cut-vertex $v$. 
Let $S$ be the vertex set of one component of $G-v$. Let $G_1=G[S \cup\{v\}]$, 
and let $G_2=G-S$. Both $G_1$ and $G_2$ are 4-cactus graphs, and every block of 
$G$ is 
a block in exactly one of $G_1$ and $G_2$, and thus each has fewer blocks 
than $G$. Hence by the induction hypothesis, we have
$|E(G_i)| \leq 2|V(G_i)|-2$  when $|V(G_i)| \equiv  1 \pmod 3$
and $|E(G_i)|\leq 2|V(G_i)|-3$ when $|V(G_i)|\equiv  0 ~\text{or}~ 2\pmod 
3$ for $i=1,2$.

Since $v$ belongs to both graphs $G_1$ and 
$G_2$,  we have $|V(G_1)|+|V(G_2)|=n+1$. We discuss three cases.

\noindent \textbf{Case 1.}  $n \equiv  0 \pmod 3$.  In this case, $|V(G_i)| 
\equiv 1 \pmod 3$ and 
$|V(G_{2-i})| \equiv 2 \pmod 3$ where $i=1$ or $i=2$,  or $|V(G_i)| \equiv 
2 \pmod 3$ for $i=1,2$.  In the first case, without loss of generality, we may 
assume that  $|V(G_1)| \equiv 1~ \pmod 3$ and 
$|V(G_{2})| \equiv 2 \pmod 3$.  We thus have
$|E(G)|=|E\left(G_1\right)|+|E\left(G_2\right)| \leq (2|V(G_1)|-2) + 
(2|V(G_2)|-3) = 2n-3.$ For the second case, we have
$|E(G)|=|E\left(G_1\right)|+|E\left(G_2\right)| \leq (2|V(G_1)|-3) + 
(2|V(G_2)|-3) = 2n-4<2n-3,$ as desired.

\noindent \textbf{Case 2.} 
 $n \equiv 1  \pmod 3$. Then $|V(G_i)| \equiv 0 \pmod 3$ and 
$|V(G_{2-i})| \equiv 2 \pmod 3$ where $i=1$ or $i=2$,  or $|V(G_i)| \equiv 
1 \pmod 3$ for $i=1,2$.  In the first case, we may 
assume that  $|V(G_1)| \equiv 0 \pmod 3$ and 
$|V(G_{2})| \equiv 2 \pmod 3$.  We thus have
$|E(G)|=|E\left(G_1\right)|+|E\left(G_2\right)| \leq (2|V(G_1)|-3) + 
(2|V(G_2)|-3) = 2n-4<2n-2$. For the second case, we have
$|E(G)|=|E\left(G_1\right)|+|E\left(G_2\right)| \leq (2|V(G_1)|-2) + 
(2|V(G_2)|-2) = 2n-2,$ as desired.

\noindent \textbf{Case 3.} 
$n \equiv 2 \pmod 3$. Then $|V(G_i)| \equiv 1\pmod 3$ and 
$|V(G_{2-i})| \equiv 2 \pmod 3$ where $i=1$ or $i=2$,  or $|V(G_i)| \equiv 
0 \pmod 3$ for $i=1,2$.  In the first case, we may assume that  $|V(G_1)| 
\equiv 1~ \pmod 3$ and $|V(G_{2})| \equiv 0 \pmod 3$ .  We thus have
$|E(G)|=|E\left(G_1\right)|+|E\left(G_2\right)| \leq (2|V(G_1)|-2) + 
(2|V(G_2)|-3) = 2n-3.$ For the second case, we have
$|E(G)|=|E\left(G_1\right)|+|E\left(G_2\right)| \leq (2|V(G_1)|-3) + 
(2|V(G_2)|-3) = 2n-4<2n-3$.

Thus we finish our proof.
\end{proof}

A $k$-cactus graph  on $n$ vertices with the largest possible number of 
edges is called  {\it extremal} for $n$ and $k$. Specifically, 
any extremal $2$-cactus with $n$ vertices has 
$\left\lfloor \frac{5(n-1)}{3}\right\rfloor$ edges.  Note that the number of 
edges of extremal $2$-cactus graphs is the upper bound  in Theorem 
\ref{2or3cactus} when $k=2$.

\begin{theorem}\label{2ex}
Let $G$ be an extremal $2$-cactus with $n\ge 1$ vertices. The following 
statements 
hold true.
\begin{itemize}
\item [(1)] If $n\equiv 0 \pmod 3$, then $G$ is obtained  by 
 coalescing  $\frac{n-3}{3}$ 
copies of $\theta(1,2,2)$ and one copy of $K_3$.
\item [(2)] If $n \equiv 1 \pmod 3$ and $n \ne 1$, then
$G$ is obtained  by coalescing $\frac{n-1}{3}$ copies of $\theta(1,2,2)$ when 
$n\ge 2$; if $n=1$, then $G\cong K_1$.
\item [(3)]  If  $n\equiv 2 \pmod 3$, then $G$ is either obtained  by 
coalescing by $\frac{n-2}{3}$ 
copies of $\theta(1,2,2)$ and a copy of $K_2$, or obtained  by 
coalescing  $\frac{n-5}{3}$ copies of $\theta(1,2,2)$ and either one 
copy of $\theta(1,2,3)$ or two copies of $K_3$.   
\end{itemize}
\end{theorem}
\begin{proof}

First, we can easily check that the number of edges of these graphs is $ 
\left\lfloor \frac{5(n-1)}{3}\right\rfloor$, which also shows that the bound 
given by Theorem \ref{2or3cactus}  for $k=2$ is tight.

 Let $G$ be an extremal 2-cactus with $n$ vertices. Let $B_1,\ldots, 
B_t$ be the blocks of $G$, where $t$ is the number of blocks of $G$.
For the sake of simplicity and smoothness in the subsequent proof, in what 
 follow, we can assume (without loss of generality) that all the blocks of $G$ 
 share a common vertex $u$,  otherwise, we can obtain a new graph with that 
 assumption 
by re-coalescing the blocks of $G$, and Lemma \ref{edgesame} implies that the 
new graph is also an extremal graph with same order.  But we need note that  
although we made the assumption on $G$,  the following claims or properties on 
$G$ hold for any extremal 2-cactus.

\begin{claim}\label{c1}
If $B$ is a block of $G$, then $B$ has the maximum possible number of edges.
\end{claim}

\begin{proof}
Clearly, if $B$ does not have the maximum number of edges, we can replace $B$ 
with another block that has the same number of vertices but more edges (while 
still being a 2-cactus graph). In this case, the number of edges of $G$ would 
increase, contradicting the premise.
\end{proof}

\begin{claim}\label{c2}
Each block of $G$ has at most $5$ vertices.
\end{claim}
\begin{proof}
Suppose that there exists a block, namely $B_1$,  with  at least 6 vertices. If 
$t=1$, i.e., $G\cong B_1$, then by Proposition \ref{2c_k}, $|E(G)|\le n+1 
<\left\lfloor\frac{5(n-1)}{3}\right\rfloor$, a contradiction. If $G$ has more 
than two blocks. Recall that $u$ is the unique cut vertex of $G$. Let 
$G'=G[V(G)\setminus V(B_1-u)$]. Clearly, $B_1$ and $G'$ are connected. 
Thus 
\begin{linenomath}
\begin{align*}
 |E(G)|&=|E(G')+|E(B_1)|\\
 &\le\left\lfloor\frac{5(|V(G')|-1)}{3}\right\rfloor+(|V(B_1)|+1)\\
 &=\left\lfloor\frac{5(|V(G')|-1)}{3}+|V(B_1)|+1\right\rfloor\\
 &=\left\lfloor  
 \frac{5(n-1)}{3}-\frac{2|V(B_1)|-8}{3}\right\rfloor.\label{2c_ex}
\end{align*}
\end{linenomath}
Since $|V(B_1)|\ge 6$, we have $\frac{2|V(B_1)|-8}{3}\ge \frac{4}{3}> 1$. 
Thus $|E(G)|\le \left\lfloor  \frac{5(n-1)}{3}\right\rfloor-1$, a contradiction.
\end{proof}

\begin{claim}\label{c3}
$G$ has at most one block with exactly $5$ vertices;  if the block exists,  it 
is a $\theta(1,2,3)$-block.
\end{claim}
\begin{proof}
Suppose that $G$ has two or more blocks  with exactly $5$ vertices.  Let  
$B_1$ and $B_2$ be such blocks. Combining  Claim \ref{c1}  with Proposition 
\ref{2c_small}  (iii), we have $|E(B_1)|=|E(B_2)|=6$.  Let $B^*=B_1\cup B_2$ 
and $G'=G[V(G)\setminus V(B^*-u)]$. Clearly,  $|V(B^*)|=9, |E(B^*)|=12$ and
$|V(G')|+|V(B^*)|=n+1$. Therefore, we have $|E(G)|=|E(G')|+|E(B^*)|
 \le\left\lfloor\frac{5(|V(G')|-1)}{3}\right\rfloor+12
 =\left\lfloor\frac{5(n-1)}{3}-\frac{4}{3}\right\rfloor 
 \le\left\lfloor\frac{5(n-1)}{3}\right\rfloor -1$, a contradiction. 
 Furthermore, by Claim \ref{c1} and  Theorem \ref{2c_k}, if there is a 
 block with exactly  $5$ vertices, then it is a $\theta(1,2,3)$-block.
\end{proof}

\begin{claim}\label{c4}
$G$ has at most two blocks with exactly $3$ vertices;  if these blocks exist,  
they are $K_3$-blocks.
\end{claim}
\begin{proof}
 By Claim \ref{c1}, any block with exactly $3$ vertices is isomorphic to $K_3$.
Suppose for a contradiction, that $G$ has three or more blocks  with  exactly  
$3$ vertices. Let  $B_1,B_2$ and $B_3$ be such blocks. Let $B^*=B_1\cup B_2 
\cup B_3$ and $G'=G[V(G)\setminus V(B^*-u)]$. Clearly,  $|V(B^*)|=7, 
|E(B^*)|=9$ and $|V(G')|+|V(B^*)|=n+1$. Therefore, we have 
$|E(G)|=|E(G')|+|E(B^*)|\le\left\lfloor\frac{5(|V(G')|-1)}{3}\right\rfloor+9
 =\left\lfloor\frac{5(n-1)}{3}\right\rfloor -1$, a contradiction. 
 \end{proof}

\begin{claim}
$G$ has at most one $K_2$-block.
\end{claim}

\begin{proof}
If there are two or more $K_2$-blocks, then we can add an edge between two 
different leaves of $G$, and the resulting new graph is  still a 
2-cactus graph, but with more edges, a contradiction. 
\end{proof}

Due to Claims 1-5, we deduce the following property of $G$.
\begin{property}\label{pro1}
 Any block  of $G$ is isomorphic to  
either $K_2$, $K_3$, $\theta(1,2,2)$, or  $\theta(1,2,3)$. Additionally, the 
number of blocks isomorphic to $K_2$ or $\theta(1,2,3)$ is at most $1$ each,  
and the number of blocks isomorphic to $K_3$ is at most $2$.
\end{property}

Now, we discuss three cases according to the remainder of $n$ divided by 3.
We rewrite the upper bound $\left\lfloor\frac{5(n-1)}{3}\right\rfloor$ as:
\begin{linenomath}
$$|E(G)|=\begin{cases}
5 l-2 &\text{if}~ n=3l, \\
5 l &\text{if}~ n=3 l+1, \\
5 l+1 &\text{if}~ n=3 l+2.
\end{cases}$$
\end{linenomath}
\vspace{0.5cm}

\noindent \textbf{Case 1.} $n=3l$. We need to prove that 
each block of $G$ is either $K_3$ or $\theta(1,2,2)$, and the number of 
$K_3$-blocks is exactly one. Due to Property \ref{pro1},  it 
suffices to show that 
$G$ does not contain $\theta (1,2,3)$ and $K_2$ as its blocks, and $G$ contain 
exactly one $K_3$. Suppose that $G$ contains a block, namely  $B_1$, which is 
isomorphic to $\theta(1,2,3)$ or $K_2$.  Let $G'=G[V(G')\setminus 
V(B_1-u)$].  If $B_1\cong \theta(1,2,3)$, then clearly, $ |V(B)|=5$ and 
$|V(G')|=n-4=3l-4=3(l-2)+2$. Then $|E(G)|=|E(G')|+|E(B_1)|\le 
(5(l-2)+1)+6=5l-3<5l-2$, a contradiction. 
Similarly, if $B_1\cong  K_2$, then  clearly $ |V(B)|=2$ and 
$|V(G')|=n-1=3l-1=3(l-1)+2$. Then $|E(G)|=|E(G')|+|E(B_1)|\le 
(5(l-1)+1)+1=5l-3<5l-2$, a contradiction. 

Now we prove that $G$ has exactly a $K_3$-block. If there does not  exist 
$K_3$-blocks, then by the previous analysis every 
block is a copy of $\theta(1,2,2)$. Then $|E(G)|$  will be a multiple of 3 plus 
1, which contradicts the premise of $s = 0$.  By Property 
\ref{pro1}, $G$ has at most two $K_3$-blocks. If there exist two blocks, namely 
$B_1$ and $B_2$, isomorphic to $K_3$, we assume that $B=B_1\cup B_2$ and  
$G'=G[V(G')\setminus V(B-u)$], then  $ |V(B)|=5$ and $|V(G')|=3(l-2)+2$.
Similarly,  $|E(G)|=|E(G')|+|E(B)|\le 5l-1$, a contradiction. Thus, we have 
proved the desired result.

\noindent \textbf{Case 2.} $n=3l+1$. Clearly, $n$=1, we have $G\cong K_1$. So 
we 
assume that $n\ge 4$. Suppose 
that there is a block of $G$, namely $B_1$, which is not
isomorphic to $\theta(1,2,2)$. Let $G'=G[V(G)\setminus 
(V(B_1-u)$]. For $B_1\cong K_2$,  note that 
$|V(G')|=3l$. Thus $|E(G)|=|E(G')|+|E(B_1)|\le (5l-2)+1=5l-1< 5l$, a 
contradiction. For $B_1\cong K_3$, now $|V(G')|=3(l-1)+2$. Thus 
$|E(G)|=|E(G')|+|E(B_1)|\le (5(l-1)+1)+3=5l-1< 5l$, a contradiction.
For $B_1\cong \theta(1,2,3)$, now $|V(G')|=(3l+1)-4=3(l-1)$. Thus, 
$|E(G)|=|E(G')|+|E(B_1)|\le (5(l-1)-2)+6=5l-1\le 5l$, a contradiction. 
Therefore, every block of $G$ is isomorphic to $\theta(1,2,2)$.

\noindent \textbf{Case 3.} $n=3l+2$.  First, it is impossible that 
every block of $G$ is isomorphic to $K_3$; otherwise, $n \equiv 1 \pmod 3$, 
which contradicts the assumption. Next, we claim that if  $G$ has
a block, namely $B_1$, isomorphic to $K_3$, then $G$ will contain exactly two 
$K_3$-blocks . Let $B_1$ then let $G'=G[V(G)\setminus 
V(B_1-u)]$. Clearly, $G'$ is also an extreme graph with $3l$ vertices.
So, by the proof of Case 1, $G'$ has exactly one $K_3$-block, and hence $G$ 
has  exactly two $K_3$-blocks.  Therefore, the following three cases (i)-(iii) 
are all possible: (i) $G$ has exactly one $K_2$-block; (ii) $G$ has
exactly one $\theta(1,2,3)$-block; (iii) $G$ has two $K_3$-blocks. Moreover, 
we  show that these three cases are mutually exclusive. Assume (i) and (ii) 
occur together, i.e.,  $G$ contains one  $K_2$-block, namely $B_1$, and one 
$\theta(1,2,3)$-block, namely $B_2$.  Then let $B=B_1\cup B_2$ and  
$G'=G[V(G)\setminus V(B-u)]$. Clearly, $|V(B)|=6$ and $|E(B)|=7$. So $|V(H)|= 
3l+2-5=3(l-1)$ and $|E(H)|\le 5(l-1)-2 =5l-7$. Hence $|E(G)|=|E(B)|+|E(G')|\le 
7+5l-7=5l<5l+1$. Similarly, if (i) and (iii), or (ii) and (iii) occur 
simultaneously, we can similarly conclude that $|E(G)|$ is less than $5l+1$, 
which contradicts the assumption.

Based on the analysis of the graph blocks in the three cases above, we have 
directly proven this theorem.
\end{proof}

The characterization theorems for extremal 3-cactus graphs  and 4-cactus graphs 
are given here. The proof techniques of the following two theorems are
similar to Theorem \ref{2ex}. For the sake of article conciseness, we  leave 
the similar and tedious proof details to the reader.

\begin{theorem}\label{3ex}
Let $G$ be an extremal $3$-cactus with $n\ge 1$ vertices. The following 
statements 
hold true.
\begin{itemize}
\item [(1)]If $n\equiv 0 \pmod 4$, then $G$ is obtained  by 
 coalescing  $\frac{n-4}{4}$ 
copies of $\theta(1,2,2,2)$ and one copy of $\theta(1,2,2)$.
\item [(2)] If $n \equiv 1 \pmod 4$ and $n\ne 1$, then $G$ is obtained  by 
coalescing $\frac{n-1}{4}$ 
copies of $\theta(1,2,2,2)$; if $n= 1$, then $G\cong K_1$.
\item [(3)]  If  $n\equiv 2 \pmod 4$, then $G$ is either obtained  by 
coalescing by $\frac{n-2}{4}$ 
copies of $\theta(1,2,2,2)$ and a copy of $K_2$, or obtained  by 
coalescing  $\frac{n-6}{4}$ copies of $\theta(1,2,2,2)$ and either one 
copy of $\theta(2,2,3)$, one copy of  $\theta(2,2,2,2)$,  or one copy of 
$\theta(1,2,2)$ and $K_3$.
\item [(4)]  If  $n\equiv 3 \pmod 4$, then $G$ is either obtained  by 
coalescing by $\frac{n-3}{4}$ 
copies of $\theta(1,2,2,2)$ and a copy of $K_3$, or obtained  by 
coalescing  $\frac{n-7}{3}$ copies of $\theta(1,2,2,2)$ and two 
copies of $\theta(1,2,2)$. 
\end{itemize}
\end{theorem}

\begin{theorem}\label{4ex}
Let $G$ be an extremal $4$-cactus with $n\ge 1$ vertices. The following 
statements 
hold true.
\begin{itemize}
\item [(1)]If $n\equiv 0 \pmod 3$, then $G$ is  obtained by 
 coalescing  $\frac{n-3}{3}$ copies of $K_4$ and one copy of $K_3$, or  
 obtained by coalescing  $\frac{n-6}{3}$ copies of $K_4$ and one copy of 
 $\theta(1,2,2,2,2)$.
\item [(2)] If $n \equiv 1 \pmod 3$ and $n\ne 1$, then $G$ is obtained  by 
coalescing $\frac{n-1}{3}$ copies of $K_4$; if $n= 1$, then $G\cong K_1$.
\item [(3)] If  $n\equiv 2 \pmod 3$, then $G$ is either obtained  by 
coalescing by $\frac{n-2}{3}$ 
copies of $K_4$ and a copy of $K_2$, or obtained  by 
coalescing  $\frac{n-5}{3}$ copies of $K_4$ and either one 
copy of $\theta(1,2,2,2)$, one copy of $\tilde{\theta}_1$, or  one copy of 
$\tilde{\theta}_2$.   
\end{itemize}
\end{theorem}

\section{Concluding remarks}\label{discussion}
In this article, we have made a start on the topic of ``extremal” $k$-cactus 
graphs.  We give  tight upper bounds on the sizes of $k$-cactus graphs when $k$ 
takes small values.
Moreover, the corresponding 
extremal graphs are also characterized.

The next step would seem to be to look at $k$-cactus graphs for  larger values 
of  $k$ (see also Problem \ref{p1}). To be frank, for relatively smaller values 
of $k$, such as $k=5$ or $6$, it can be handled by applying the similar 
approach in this article.  Patience is needed,  but the level of complexity 
should not 
increase 
too much. However, for larger or especially general values of $k$, establishing 
(sharp) bounds is likely to be a difficult problem. It is worth noting that we 
first characterize 
$2$-, $3$-, $4$-cactus graphs, and then give upper bounds on their sizes. These 
characterizations heavily rely on the tool of the ear decomposition. However, 
as the value of $k$ in  $k$-cactus graphs increases, similar characterizations 
of their blocks
become an increasingly difficult thing, because the choices of adding ears   
will grow significantly (even when considering symmetry). As a result, we might 
need to find some cleverer ways to  resolve these problems without using the 
ear decomposition theorem.

Fortunately, we can still roughly catch some potential  strategies to tackle 
the extremal problem  on  the general $k$-cactus graphs. From our theorems, it 
appears that the primary constituents of extremal $k$-cactus graphs are 
(2-connected) extremal nice $k$-cactus graphs with minimum order. Therefore, 
for solving Problem \ref{p1}, we may need to 
solve two preliminary problems below.

\begin{problem}
What are the (tight) lower bounds on the orders of nice $k$-cactus graphs?
\end{problem}

\begin{problem}
What are  the (tight) upper bounds on the sizes of $2$-connected nice 
$k$-cactus graphs with $n$ vertices when $n\le k+1$?
\end{problem}

\section{Acknowledgment}
We claim that there is no conflict of interest in our paper.
No data was used for the research described in the article.


\begin{thebibliography}{99}



\bibitem{Araujo}
J..Araujo, F. Havet,  C. L. Sales, A. Silva,  Proper orientation of cacti, 
Theor. Comput. Sci., 639 (2016), 14-25.

\bibitem{Ben-Moshe}

B. Ben-Moshe et al., Efficient algorithms for the weighted 2-center problem in 
a cactus graph, Lect. Notes Comput. Sci., 3827 (2005) 693--703.


\bibitem{Czap}
J. Czap,  A note on odd facial total-coloring of cacti, Commun. Comb. Optim., 8 
(2023) 589-594.



\bibitem{El-Mallah}
E.S. El-Mallah, C.J. Charles, The complexity of some edge deletion problems, 
IEEE Trans. Circuits Syst., 35 (1988) 354-362.





\bibitem{Erdos}
P. Erd\H{o}s,  Elementary problem E3255, Amer. Math. Monthly, 95 (1988), 259.


\bibitem{Harary}
F. Harary,  G.E. Uhlenbeck, On the number of Husimi trees: I, Proc. Natl. Acad. 
Sci. USA, 39 (1953) 315-322.











\bibitem{Li2019}

S.C. Li, L.C. Zhang, M.J. Zhang, On the extremal cacti of given parameters with 
respect to the difference of zagreb indices, J. Comb. Optim.,  38 (2019) 
421-442.








\bibitem{Lovasz}
 L. Lov{\'a}sz, Combinatorial Problems and Exercises. Second Edition. 
 North-Holland Publishing Co., Amsterdam, 1993.



\bibitem{Rosa}
 A. Rosa, Cyclic Steiner triple systems and labelings of triangular cacti, 
 Scientia, 1 (1988) 87–95.


\bibitem{Sedlar}
J. Sedlar, R. Sedlar, Vertex and edge metric dimensions of cacti, Discrete 
Appl. Math., 320 (2022) 126-139.





\bibitem{Toida1973}
S. Toida, Properties of a Euler graph, J. Franklin Inst., 295 (1973) 
343–34.




\bibitem{west2001}
D.B. West, Introduction to {Graph Theory} (second edition), Prentice Gall, 
2001.


\bibitem{whitney}
H. Whitney, Congruent graphs and the connectivity of graphs, Amer. J. Math., 54 
(1932) 150-168.

\bibitem{Xue2023}
 Y. Xue, B.T. Yang, S. Zilles, L.S. Wang, Fast searching on cactus graphs, J. 
 Comb. Optim., 45 (2023) 22.



\end{thebibliography}
\end{document}